\documentclass[12pt]{amsart}
\usepackage{amsfonts}
\usepackage{amsmath}
\usepackage{amssymb}
\usepackage{color}
\usepackage{graphicx}
\usepackage{xspace}
\usepackage{tikz}

 \font \sevenrm=cmr7
 \input xy
\xyoption{all}

 \newcommand{\nc}{\newcommand}

\newtheorem{thm}{Theorem}[section]
\newtheorem{exam}[thm]{Example}

\newtheorem{cor}[thm]{Corollary}
\newtheorem{con}[thm]{Conjecture}
\newtheorem{lem}[thm]{Lemma}
\newtheorem{prop}[thm]{Proposition}
\newtheorem{defn}[thm]{Definition}
\newtheorem{rmk}[thm]{Remark}

\nc{\ignore}[1]{{}}
\nc{\mrm}[1]{{\rm #1}}
\nc{\dirlim}{\displaystyle{\lim_{\longrightarrow}}\,}
\nc{\invlim}{\displaystyle{\lim_{\longleftarrow}}\,}
\nc{\vep}{\varepsilon} \nc{\ep}{\epsilon}
\nc{\sigmat}{\widetilde\sigma}
\nc{\ostar}{\overline{*}}
\nc{\mchar}{\mrm{Char}}
\nc{\Hom}{\mrm{Hom}}
\nc{\id}{\mrm{id}}

\nc{\remark}{\noindent{\bf{Remark:}}}
\nc{\remarks}{\noindent{\bf{Remarks:}}}

 \nc{\delete}[1]{}
 \nc{\grad}[1]{^{({#1})}}
 \nc{\fil}[1]{_{#1}}

\nc{\BA}{{\Bbb A}} \nc{\CC}{{\Bbb C}} \nc{\DD}{{\Bbb D}}
\nc{\EE}{{\Bbb E}} \nc{\FF}{{\Bbb F}} \nc{\GG}{{\Bbb G}}
\nc{\HH}{{\Bbb H}} \nc{\LL}{{\Bbb L}} \nc{\NN}{{\Bbb N}}
\nc{\PP}{{\Bbb P}} \nc{\QQ}{{\Bbb Q}} \nc{\RR}{{\Bbb R}}
\nc{\TT}{{\Bbb T}} \nc{\VV}{{\Bbb V}} \nc{\ZZ}{{\Bbb Z}}
\nc{\Cal}[1]{{\mathcal {#1}}}
\nc{\mop}[1]{\mathop{\hbox {\rm #1} }\nolimits}
\nc{\smop}[1]{\mathop{\hbox {\sevenrm #1} }\nolimits}
\nc{\mopl}[1]{\mathop{\hbox {\rm #1} }\limits}
\nc{\frakg}{{\frak g}}
\nc{\g}[1]{{\frak {#1}}}
\def \restr#1{\mathstrut_{\textstyle |}\raise-8pt\hbox{$\scriptstyle #1$}}
\def \srestr#1{\mathstrut_{\scriptstyle |}\hbox to
  -1.5pt{}\raise-4pt\hbox{$\scriptscriptstyle #1$}}
\nc{\wt}{\widetilde}
\nc{\wh}{\widehat}
\nc{\un}{\hbox{\bf 1}}
\nc{\redtext}[1]{\textcolor{red}{\tt #1}}
\nc{\bluetext}[1]{\textcolor{blue}{#1}}
\nc{\comment}[1]{[[{\tt {#1}}]] }
\nc{\R}{{\mathbb R}}

\nc\fleche[1]{\mathop{\hbox to #1 mm{\rightarrowfill}}\limits}
\def\semi{\mathrel{\times}\kern -.85pt\joinrel\mathrel{\raise 1.4pt\hbox{${\scriptscriptstyle |}$}}}

\def\fleche#1{\mathop{\hbox to #1 mm{\rightarrowfill}}\limits}
\def\gfleche#1{\mathop{\hbox to #1 mm{\leftarrowfill}}\limits}
\def\inj#1{\mathop{\hbox to #1 mm{$\lhook\joinrel$\rightarrowfill}}\limits}
\def\ginj#1{\mathop{\hbox to #1 mm{\leftarrowfill$\joinrel\rhook$}}\limits}
\def\surj#1{\mathop{\hbox to #1 mm{\rightarrowfill\hskip 2pt\llap{$\rightarrow$}}}\limits}
\def\gsurj#1{\mathop{\hbox to #1 mm{\rlap{$\leftarrow$}\hskip 2pt \leftarrowfill}}\limits}

\begin{document}

%\pagestyle{empty}

%\thispagestyle{empty}

%%%%%%%%%%%%%%%%%%%%%%%%%%%%%%%%%%%%%%%%%%%%%
%%%%%%%%%%%%%%%%%%%%%%%%%%%%%%%%%%%%%%%%%%%%%

\title[Sewing lemma and knitting lemma for metric spaces]{Sewing lemma and knitting lemma for metric spaces} 
\author{C.~Curry, D.~Manchon}
\address{C. Curry: NTNU Gj\o vik, charles.curry@ntnu.no}
\address{D. Manchon: Laboratoire de Math\'ematiques Blaise Pascal (LMBP), CS 60026, 3 place Vasar\'ely, 63178 Aubi\`ere, France. Dominique.Manchon@uca.fr}
\begin{abstract}
We state and prove a sewing lemma in the general context of families of complete metric spaces indexed by an interval of the real line, encompassing the flow sewing lemma proved by I. Bailleul in 2015. A further generalisation to other metric parameter spaces $P$ than intervals is moreover proposed, leading to a representation of the groupoid of thin-equivalent Lipschitz paths on $P$. Under a stronger hypothesis, we finally prove a two-dimensional version, the knitting lemma, which gives rise to a representation of the Lipschitz homotopy groupoid of the parameter space, without thinness condition.
\end{abstract}
\maketitle

\tableofcontents

%%%%%%%%%%%%%%%%%%%%%%%%%%%%%%%%%%%%%%%%%%%%%

%%%%%
\section{Introduction}
%%%%%

The sewing lemma first appeared in the work of Feyel and de la Pradelle \cite{FP2006}, and was clarified by M. Gubinelli \cite{Gubi2010}. The statement is that, given a Banach space valued function $\mu(s,t)$ obeying 
\[|\mu(s,t)-\mu(s,u)-\mu(u,t)| \leq \sum_{i=1}^n C_i |t-u|^{a_i}|u-s|^{b_i} \hbox{ with }a_i+b_i={1+\varepsilon},\]
there exists a unique (up to additive constant) $\phi(t)$ such that $|\phi(t)-\phi(s)-\mu(s,t)|\leq C|t-s|^{1+\varepsilon}$.
The intended application was to show that integrals $I(s,t) = \phi(t)-\phi(s)$ could be defined through Riemann sums $\sum_i \mu(t_{i},t_{i+1})$ in a limit where the mesh of partition of $[s,t]$ tends to zero. In this the inspiration was Lyons' theorem on almost multiplicative functionals, used to extend Young integration to rough paths.\\

Shortly after, Feyel, de la Pradelle and Mokobodzki \cite{FPM2008} proved a non-commutative sewing lemma for functions $\mu(s,t)$ taking values in an associative monoid equipped with a complete metric; a unique $\phi(t)$ exists under the hypothesis 
\[d\big(\mu(s,t),\, \mu(s,u)\circ\mu(u,t)\big)\leq \sum_{i=1}^n C_i |t-u|^{a_i}|u-s|^{b_i} \hbox{ with }a_i+b_i={1+\varepsilon}.\]
A similar result was proved by Bailleul \cite{B2015} and used in the construction of flows driven by rough paths. Subsequently there have appeared several further generalizations \cite{BL2018a, BL2018b, D2018} of non-commutative sewing lemmas.\\

Hitherto all sewing lemmas have been based on functions of two real variables $s,t\in\mathbb{R}\times\mathbb{R}$, indeed since the work of Gubinelli on rough paths \cite{Gubi2004} it has been possible to interpret such constructions in the language of cochain complexes. On examining the proofs of existing sewing lemmas, however, it becomes clear that it is not the order structure of $\mathbb{R}$ alone that facilitates the proof, but rather a groupoid structure that permits concatenation of paths and flows. \\

\noindent In the present work we

\begin{enumerate}

\item State and prove a sewing lemma (Theorem \ref{thm:main}) in the general context of complete metric spaces, encompassing all of the variants mentioned above. 

\item Propose in Paragraph \ref{sect:gen-sewing} a conjectural generalization of the sewing lemma to more general parameter spaces $P$ than the real line, involving the groupoid of thin-equivalent Lipschitz paths on the metric space $P$.

\item Prove in Paragraph \ref{sect:knitting} a \textsl{knitting lemma} for general metric parameter spaces $P$. This result, which holds true under a much stronger condition than what is necessary for the sewing lemma, can be briefly described as a two-dimensional version of the latter, and leads to a representation of the Lipschitz homotopy groupoid (Theorem \ref{thm:knitting-lemma}).
\end{enumerate}
\vskip 3mm
\noindent \textbf{Acknowledgements:} We thank the two anonymous referees for their careful reading and most valuable suggestions.
%%%%%
\section{Extended metric spaces of continuous maps}\label{metric}
%%%%%

%%%
\subsection{Extended metric spaces}
%%%

\begin{defn}{\rm(\cite[Pages 419--420]{AP1956}, see also \cite[Paragraph 1.2.3]{BL2018a} and \cite{BBI2001})}
An \textbf{extended metric space} is a set $E$ endowed with a distance $d:E\times E\to \mathbb R^+\cup\{\infty\}$, such that the usual axioms for a distance, i.e.
\begin{itemize}
\item $d(x,y)=0$ if and only if $x=y$,
\item $d(x,y)=d(y,x)$,
\item triangle inequality $d(x,z)\le d(x,y)+d(y,z)$
\end{itemize}
are verified. The only difference with a usual metric space is that the infinite value for the distance is allowed. The \textbf{galaxy} of $x\in E$ is the subset of points at finite distance from $x$.
\end{defn}

Any galaxy is a metric space in the usual sense, and any extended metric space is a disjoint (and disconnected) union of galaxies. The notion of completeness can be straightforwardly adapted to extended metric spaces, because all terms in a Cauchy sequence belong to the same galaxy, except perharps a finite number of those.

%%%
\subsection{Extended metric spaces of continuous maps}
%%%

Let $X$ and $Y$ be two metric spaces (in the extended sense or not), where $Y$ is complete. The set $C_{YX}$ of continuous maps from $X$ to $Y$ carries a natural structure of complete extended topological space, given by the distance:
\begin{equation}
	d_{YX}(f,g):=\mop{sup}_{x\in X}d_Y\big(f(x),g(x)\big).
\end{equation}
The triangle inequality is obviously verified, and the completeness is verified as follows: if $(f_k)_{k\ge 1}$ is a Cauchy sequence in $C_{YX}$, then $\big(f_k(x)\big)_{k\ge 1}$ is a Cauchy sequence in $Y$ for any $x\in X$, which thus converges to an element $f(x)$ of $Y$. This defines the evaluation at $x$ of the limit $f$ of the sequence $(f_k)$.
\begin{rmk}\rm
Note that, even when $X$ and $Y$ are metric spaces in the usual sense, it may be not the case for $C_{YX}$. For example, the identity $x\mapsto x$ and any constant function $x\mapsto x_0$ are infinitely distant from each other in $C_{XX}$ whenever $X$ is unbounded.
\end{rmk}
\noindent Recall that a map $f:X\to Y$ between two metric spaces is Lipschitz if $\mop{Lip}(f)<+\infty$, with
\begin{equation}
	\mop{Lip}(f):=\mop{sup}_{x'\neq x''}\frac{d_Y\big(f(x'),f'x'')\big)}{d_X(x',x'')}.
\end{equation}
In particular any Lipschitz map is continuous. If $Z$ is another metric space and $g:Y\to Z$ is another map, we obviously have
\begin{equation}\label{lip-comp}
	\mop{Lip}(g\circ f)\le\mop{Lip}(g)\mop{Lip}(f).
\end{equation}
The two following lemmas can be found in \cite[Paragraph 2.1]{D2018}. We simply rephrase them  in the context of several metric spaces.

\begin{lem}\label{liptwo}
Let $X,Y,Z$ be three metric spaces. Let $f,f'\in C_{YX}$ and $g,g'\in C_{ZY}$. Then
\begin{equation}
	d_{ZX}(g\circ f, g'\circ f')\le d_{ZY}(g,g')+\mop{Lip}(g')d_{YX}(f,f').
\end{equation}
\end{lem}

\begin{proof}
From the triangle inequality in $C_{ZX}$ we have:
\begin{eqnarray*}
d_{ZX}(g\circ f,g'\circ f')&\le& d_{ZX}(g\circ f,g'\circ f)+d_{ZX}(g'\circ f,g'\circ f')\\
&\le &\mop{sup}_{x\in X}d_Z\big(g\circ f(x),g'\circ f(x)\big)+\mop{sup}_{x\in X}d_Z\big(g'\circ f(x),g'\circ f'(x)\big)\\
&\le &\mop{sup}_{y\in Y}d_Z\big(g(y),g'(y)\big)+\mop{Lip}(g')\mop{sup}_{x\in X}d_Y\big(f(x),f'(x)\big)\\
&\le &d_{ZY}(g,g')+\mop{Lip}(g')d_{YX}(f,f').
\end{eqnarray*}
\end{proof}

\noindent Lemma \ref{liptwo} can be immediately generalised as follows:
\begin{lem}\label{lipmany}
Let $X_0,\ldots,X_r$ be a collection of metric spaces, with $X_j$ complete for $j=0,\ldots, r-1$. Let $f_j,f'_j\in C_{X_{j-1}X_{j}}$ for any $j=1,\ldots, r$. Then
\begin{equation}
	d_{X_0X_r}(f_1\circ\cdots\circ f_r,\, f'_1\circ\cdots\circ f'_r)\le \sum_{j=1}^r
	\left(\prod_{i=1}^{j-1}\mop{Lip}(f'_i)\right)d_{X_{j-1}X_{j}}(f_j,f'_j).
\end{equation}
\end{lem}
The various distances $d_X,d_Y,d_{YX},\ldots$ involved will occasionally be denoted by the same letter $d$ when the context clearly indicates which distance is to be considered.
%%%
\subsection{Length spaces}\label{sect:length}
%%%
Let $P$ be an extended metric space, and let $\gamma:I\to P$ a continuous path. The length of the path $\gamma$ is defined by
\begin{equation}
\ell(\gamma):=\mopl{sup}_{k\in\mathbb N,\, t_0\le \cdots\le t_k \in I}\sum_{i=1}^kd\big(\gamma(t_{i-1}),\, \gamma(t_i)\big).
\end{equation}
The metric space $P$ is a length space if for any $x,y\in P$ the following equality holds:
\begin{equation}
d(x,y)=\mopl{inf}_{\smop{paths} \gamma:[0,1]\to P,\, \gamma(0)=x \smop{and} \gamma(1)=y}\ell(\gamma).
\end{equation}
In other words, a length space is an extended metric space in which the distance is given by the minimum of the lengths of paths from $x$ to $y$. More details on length spaces can be found in \cite[Paragraph 2.3]{BBI2001}.
%%%%%
\section{Local flows, local approximate flows and the sewing lemma}
%%%%%
\begin{defn}\label{localflow}
Let $(M_t,d_t)_{t\in \mathbb R}$ a one-parameter family of complete metric spaces, possibly in the extended sense. Let $C_{st}$ be the set of continuous maps from $M_t$ to $M_s$, endowed with the distance $d_{M_sM_t}$ that will be denoted by $d_{st}$. A \textbf{local flow} associated with these data is a continuous action of the pair groupoid $\mathbb R\times\mathbb R$ on the family $(M_t)$. It is therefore a two-parameter family $(\varphi_{st})_{s,t\in\mathbb R}$ of maps such that
\begin{itemize}
\item $\varphi_{st}\in C_{st}$,
\item $\varphi_{ss}=\mop{Id}_{M_s}$ for any $s\in \mathbb R$,
\item $\varphi_{st}=\varphi_{su}\circ\varphi_{ut}$ for any $s,t,u\in\mathbb R$.
\end{itemize}
\end{defn}

\noindent In particular, if $(\varphi_{st})$ is a local flow, $\varphi_{st}:M_t\to M_s$ is a homeomorphism, with inverse $\varphi_{ts}$.

\begin{rmk}\rm
In view of our main result (Theorem \ref{thm:main} below), it is natural to ask for Lipschitz bounds on the homeomorphisms $\varphi_{st}$, thus taking the metric structure of the real line $\mathbb R$ into account.
\end{rmk}

\begin{rmk}\rm
Definition \ref{localflow} together make sense for any metric space $P$ instead of $\mathbb R$, we'll return to this point later.
\end{rmk}

\begin{defn}\label{local-approx}
Let $(M_t,d_t)_{t\in \mathbb R}$ a one-parameter family of complete metric spaces, possibly in the extended sense. A \textbf{local approximate flow} associated with these data is a two-parameter family $(\mu_{st})_{s,t\in\mathbb R}$ of maps such that
\begin{itemize}
\item $\mu_{st}\in C_{st}$,
\item $\mu_{ss}=\mop{Id}_{M_s}$ for any $s\in \mathbb R$,
\item There exists a continuous function $f:\mathbb R\to[0,+\infty[$ with $f(0)=0$, such that
\[\mop{Lip}(\mu_{st})\le 1+f(t-s),\]
\item There exists $\varepsilon>0$, an integer $n\ge 1$ and two collections of positive constants $a_i,b_i, i=1,\ldots,n$ with $a_i+b_i=1+\varepsilon$, such that for any $s,t,u\in \mathbb R$ the following estimates hold:
\begin{equation}\label{laf-ci}
d_{st}\big(\mu_{st},\mu_{su}\circ\mu_{ut}\big)\le \sum_{i=1}^n C_i|t-u|^{a_i}|u-s|^{b_i}.
\end{equation}
\item There exists a non-decreasing continuous function $g:[0,+\infty[\to [1,+\infty[$ such that, for any $s,t\in\mathbb R$ and any collection $I=(t_0,\ldots,t_k)$ with $t_0=s$ and $t_k=t$, the map $\mu_{st}^I:=\mu_{t_0t_1}\circ\mu_{t_1t_2}\circ\cdots\circ\mu_{t_{k-1}t_k}$ satisfies
\[\mop{Lip}\mu_{st}^I\le g\left(\sum_{j=1}^k|t_j-t_{j-1}|\right).\]
\end{itemize}
\end{defn}
In particular, if the last item of Definition \ref{local-approx} is verified, for any subdivision $I=(t_0,\ldots,t_k)$ of the interval between $s$ and $t$, arranged in increasing or decreasing order according to the sign of $t-s$, with $t_0=s$ and $t_k=t$, the map $\mu_{st}^I$ satisfies
\[\mop{Lip}\mu_{st}^I\le g(|t-s|).\]
\begin{rmk}\rm
The last item of Definition \ref{local-approx} is seemingly more restrictive in \cite[Theorem 1]{B2015}, where the function $g$ in only defined on $[0,\delta]$ for some small $\delta\ge 0$. In view of \eqref{lip-comp}, this function can however be extended to any non-negative number by
\[g(x):=g(\delta)^{\lfloor\frac x\delta\rfloor}g\left(x-\delta\lfloor\frac x\delta\rfloor\right).\]
\end{rmk}
\begin{thm}\label{thm:main}
Let $(\mu_{st})_{s,t\in\mathbb R}$ be a local approximate flow as in Definition \ref{local-approx}. Then there exists a unique local flow $(\varphi_{st})$ such that
\begin{equation}\label{estimate-phi}
d_{st}(\varphi_{st},\mu_{st})\le C'|t-s|^{1+\varepsilon},
\end{equation}
which is obtained as the limit in $C_{st}$ of the compositions
$$
	\mu_{st_1}\circ\mu_{t_1t_2}\circ\cdots\circ\mu_{t_{k-1}t}
$$
when the mesh of the subdivision $(s,t_1,\ldots, t_{k-1},t)$ of the interval between $s$ and $t$ tends to zero. One possible constant $C'$ is given by
$$
	C'=2^{1+\varepsilon}g(|t-s|)\left(\sum_{i=1}^nC_i\right)\zeta(1+\varepsilon),
$$
where $\zeta$ is the Riemann zeta function. Moreover, the local flow $(\varphi_{st})$ satisfies the Lipschitz condition:
\begin{equation}
\mop{Lip}(\varphi_{st})\le g(|t-s|).
\end{equation}
\end{thm}

The proof of Theorem \ref{thm:main} requires several steps.
\ignore{
\begin{lem}\label{maxmax}
Let $\varepsilon >0$ and let $a\in [0,1+\varepsilon]$. Then for any $\alpha\in [0,1]$, we have
\begin{equation}
\alpha^a(1-\alpha)^{1+\varepsilon-a}\le \frac{1}{2^{1+\varepsilon}}.
\end{equation} 
\end{lem}
\begin{proof}
Considering $f_a(\alpha):=\alpha^a(1-\alpha)^{1+\varepsilon-a}$ we have $\mop{Log}f_a(\alpha)=a\mop{Log}\alpha+(1+\varepsilon-a)\mop{Log}(1-\alpha)$, hence
$$\frac{d}{d\alpha}\mop{Log}f_a(\alpha)=\frac{a}{\alpha}-\frac{1+\varepsilon-a}{1-\alpha}=\frac{a-(1+\varepsilon)\alpha}{\alpha(1-\alpha)}.$$
Hence the maximum of $f_a$ is reached for $\alpha=a/(1+\varepsilon)$, where we have:
\begin{equation}
g(a):=f_a\left(\frac{a}{1+\varepsilon}\right)=\frac{a^a(1+\varepsilon-a)^{1+\varepsilon-a}}{(1+\varepsilon)^{1+\varepsilon}}.
\end{equation}
From $\mop{Log} g(a)=a\mop{Log}a+(1+\varepsilon-a)\mop{Log}(1+\varepsilon-a)-(1+\varepsilon)\mop{Log}(1+\varepsilon)$ we get
$$\frac{d}{da}\mop{Log}g(a)=\mop{Log} a -\mop{Log}(1+\varepsilon-a),$$
which vanishes only at $a=(1+\varepsilon)/2$, for which we have
$$g\left(\frac{1+\varepsilon}{2}\right)=\frac{1}{2^{1+\varepsilon}}.$$
\end{proof}
}
Let $(\mu_{st})$ be a local approximate flow. Let $I=(t_0,\ldots,t_k)$ be a subdivision of the interval between $s$ and $t$, arranged in increasing or decreasing order according to the sign of $t-s$, with $t_0=s$ and $t_k=t$. The mesh of the subdivision is defined by:
\begin{equation}
\mop{mesh}(I):=\mop{sup}_{j=1,\ldots, k}{|t_j-t_{j-1}}|.
\end{equation}
\begin{lem}\label{lem:constant-K}
There exists a constant $K$ independent of the subdivision $I$ such that
\begin{equation}\label{constant-K}
d_{st}(\mu_{st},\mu_{st}^I)\le Kg(|t-s|)|t-s|^{1+\varepsilon}.
\end{equation}
\end{lem}
\begin{proof}
From $I$ above we define $I_1,\ldots,I_k$ as $I_j:=[t_{j-1},t_j[$ if $s<t$, and  $I_j:=]t_{j},t_{j-1}]$ if $s>t$. Denote by $|I_j|=|t_j-t_{j-1}|$ the length of the $j$th interval. We have then
\begin{equation}\label{twobytwo}
\mop{inf}_{j=1,\ldots,k-1}(|I_j|+|I_{j+1}|)\le \frac{2|t-s|}{k-1},
\end{equation}
as we can see straightforwardly by writing the whole interval between $s$ and $t$ (with $t$ excluded) as the disjoint union of the $I_j$'s, which in turn can be regrouped two by two, omitting $I_k$ if $k$ is odd. Let us consider the index $j$ realising the minimum in \eqref{twobytwo}, and let $J$ be the subdivision obtained from $I$ by merging the two blocks $I_j$ and $I_{j+1}$, thus suppressing $t_j$. We have then
\begin{eqnarray*}
d(\mu_{st}^I,\,\mu_{st}^J)&=&d(\mu_{t_0t_1}\circ\cdots\circ\mu_{t_{j-1}t_{j}}\circ\mu_{t_jt_{j+1}}\circ\cdots\circ\mu_{t_{k-1}t_k},\, \mu_{t_0t_1}\circ\cdots\circ\mu_{t_{j-1}t_{j+1}}\circ\cdots\circ\mu_{t_{k-1}t_k})\\
&\le&\mop{Lip}(\mu_{t_0t_1}\circ\cdots\circ\mu_{t_{j-2}t_{j-1}})\sum_{i=1}^n C_i|I_j|^{b_i}|I_{j+1}|^{a_i}\\
&\le &g(|t-s|)\left(\sum_{i=1}^nC_i\right)\left(\frac{2}{k-1}\right)^{1+\varepsilon}|t-s|^{1+\varepsilon}.
\end{eqnarray*}
Iterating this process down to the trivial subdivision, we get \eqref{constant-K}, with
\begin{equation}\label{constant-K-expl}
K=2^{1+\varepsilon}\left(\sum_{i=1}^nC_i\right)\sum_{\ell=1}^{\infty}\left(\frac{1}{\ell}\right)^{1+\varepsilon}=2^{1+\varepsilon}\left(\sum_{i=1}^nC_i\right)\zeta(1+\varepsilon).
\end{equation}
\end{proof}
\begin{lem}\label{mesh}
Let $(\mu_{st})$ be a local approximate flow, and let $I,J$ be two subdivisions of the interval between $s$ and $t$, with $J$ finer than $I$. Then
\begin{equation}
d(\mu_{st}^I,\,\mu_{st}^J)\le Kg(|t-s|)g\big(\mop{mesh}(I)\big)\mop{mesh}(I)^\varepsilon|t-s|.
\end{equation}
\end{lem}
\begin{proof}
We denote by $J_j$ the subdivision of the block $I_j$ induced by $J$, and we use the shorthand $\mu_{t_{j-1}t_j}^J$ for $\mu_{t_{j-1}t_j}^{J_j}$. We can compute, using Lemma \ref{lem:constant-K}:
\begin{eqnarray*}
d(\mu_{st}^I,\,\mu_{st}^J)&\le &\sum_{j=0}^{r-1}d(\mu_{t_0t_1}\cdots\mu_{t_{j-1}t_{j}}\mu_{t_jt_{j+1}}\cdots\mu_{t_{k-1}t_k},\,\mu_{t_0t_1}\cdots\mu_{t_{j-1}t_{j}}\mu_{t_jt_{j+1}}^J\cdots\mu_{t_{k-1}t_k}^J)\\
&\le &g(|t-s|)\sum_{j=0}^{r-1}d(\mu_{t_jt_{j+1}},\, \mu_{t_jt_{j+1}}^J)\\
&\le &g(|t-s|)\sum_{j=0}^{r-1}Kg(|t_{j+1}-t_{j}|)|t_{j+1}-t_{j}|^{1+\varepsilon}\\
&\le& Kg(|t-s|)g\big(\mop{mesh}(I)\big)\sum_{j=0}^{r-1}|t_{j+1}-t_{j}||t_{j+1}-t_{j}|^\varepsilon\\
&\le &Kg(|t-s|)g\big(\mop{mesh}(I)\big)\sum_{j=0}^{r-1}|t_{j+1}-t_{j}|\mop{mesh}(I)^\varepsilon\\
&\le&Kg(|t-s|)g\big(\mop{mesh}(I)\big)\mop{mesh}(I)^\varepsilon |t-s|.
\end{eqnarray*}
\end{proof}
\begin{cor}
For any pair $(I,I')$ of subdivisions we have:
\begin{equation}
d(\mu_{st}^I,\,\mu_{st}^{I'})\le Kg(|t-s|)\Big(g\big(\mop{mesh}(I)\big)\mop{mesh}(I)^\varepsilon + g\big(\mop{mesh}(I')\big)\mop{mesh}(I')^\varepsilon\Big)|t-s|.
\end{equation}
\end{cor}
\begin{proof}
It is an immediate consequence of the triangle inequality
$$d(\mu_{st}^I,\,\mu_{st}^{I'})\le d(\mu_{st}^I,\,\mu_{st}^J)+ d(\mu_{st}^J,\,\mu_{st}^{I'})$$
where $J$ is the joint subdivision of $I$ and $I'$, i.e. the coarsest subdivision which is both finer than $I$ and finer than $I'$.
\end{proof}
\begin{cor}\label{limit}
For any $s,t\in\mathbb R$, the family $(\mu_{st}^I)_I$ converges in $C_{st}$ to a limit $\varphi_{st}$ when $\mop{mesh}(I)$ tends to 0, and we have
\begin{equation}
d_{st}(\mu_{st},\varphi_{st})\le Kg(|t-s|)|t-s|^{1+\varepsilon}.
\end{equation}
where the constant $K$ is given by \eqref{constant-K-expl}.
\end{cor}

\noindent{\it Proof of Theorem \ref{thm:main}.} Let $s,t,u\in\mathbb R$, with $u$ between $s$ and $t$. For any subdivision $I$ (resp. $I'$) of the interval between $s$ and $u$ (resp. between $u$ and $t$), the union of both subdivisions yields a subdivision $J$ of the interval between $s$ and $t$, with $\mop{mesh}(J)=\mop{sup}\big(\mop{mesh}(I),\, \mop{mesh}(I')\big)$, and we have
$$\mu_{st}^J=\mu_{su}^I\circ\mu_{ut}^{I'}.$$
Letting the mesh of $I$ and $I'$ go to zero, Corollary \ref{limit} yields:
\begin{equation}
\varphi_{st}=\varphi_{su}\circ\varphi_{ut}.
\end{equation}

It remains to show that $\varphi_{ts}$ is the inverse of $\varphi_{st}$. Choosing, for example, the regular subdivision $I_k=(t_0,\ldots,t_k)$ with $t_j=s+j(t-s)/k$, and denoting by $\overline{I_k}$ the corresponding reverse subdivision, we have
\begin{eqnarray*}
\mu_{st}^{I_k}\circ\mu_{ts}^{\overline{I_k}}&=&\mu_{t_0t_1}\circ\cdots\circ\mu_{t_{k-1}t_k}\circ\mu_{t_kt_{k-1}}\circ\cdots\circ\mu_{t_1t_0}.
\end{eqnarray*}
Hence we get
\begin{eqnarray*}
d(\mu_{st}^{I_k}\circ\mu_{ts}^{\overline{I_k}},\, \mu_{ss})&\le&\sum_{j=0}^{k-1} d(\mu_{t_0t_1}\circ\cdots\circ\mu_{t_{j-1}t_j}\circ\mu_{t_jt_{j-1}}\circ\cdots\circ\mu_{t_1t_0},\, \mu_{t_0t_1}\circ\cdots\circ\mu_{t_{j}t_{j+1}}\circ\mu_{t_{j+1}t_{j}}\circ\cdots\circ\mu_{t_1t_0})\\
&\le &g(|t-s|)\sum_{j=0}^{k-1}d(\mu_{t_jt_j},\,\mu_{t_{j}t_{j+1}}\circ\mu_{t_{j+1}t_{j}})\\
&\le &g(|t-s|)k\left(\sum_{i=1}^n C_i\right)\left(\frac{|t-s|}{k}\right)^{1+\varepsilon},
\end{eqnarray*}
which gives $\varphi_{st}\circ\varphi_{ts}=\varphi_{ss}$ by letting $k$ go to infinity. Finally, the Lipschitz condition for $\varphi$ immediately follows from the last item of Definition \ref{local-approx}.

\hfill\qed
\begin{rmk}\rm
In the particular case when the function $f$ in Definition \ref{local-approx} is given by $f(\delta):=L\delta$, the last item is automatically verified with $g(\delta)=e^{L\delta}$. The very reason is that the inequality
\[\prod_{j=1}^k (1+L\delta_j)\le e^{L\delta}\]
holds for any $L\ge 0$, for any positive integer $k$ and for any $k$-tuple $(\delta_1,\ldots,\delta_k)$ of non-negative real numbers such that $\sum_{j=1}^k\delta_j=\delta$. Details are left to the reader.
\end{rmk}
\begin{rmk}\label{four-point}
\rm If $(\mu_{st})_{s,t\in\mathbb R}$ is an approximate flow in the sense of Definition \ref{local-approx}, for any $s,t,u,v\in\mathbb R$ we have from Lemma \ref{liptwo} and \eqref{laf-ci}:
\begin{eqnarray}
d_{st}(\mu_{su}\circ\mu_{ut},\, \mu_{sv}\circ\mu_{vt})&\le &d_{st}(\mu_{su}\circ\mu_{ut},\, \mu_{su}\circ\mu_{uv}\circ\mu_{vt})+d_{st}(\mu_{su}\circ\mu_{uv}\circ\mu_{vt},\, \mu_{sv}\circ\mu_{vt})\nonumber\\
&\le& \big(1+f(|u-s|)\big)\sum_{i=1}^n C_i|t-v|^{a_i}|u-v|^{b_i}+\sum_{i=1}^nC_i |s-u|^{b_i}|u-v|^{a_i}.\label{ineq-suvt}
\end{eqnarray}
Applying the inequality \eqref{ineq-suvt} above for $v=t$ gives back \eqref{laf-ci}. It gives moreover the optimal bound $0$ for $u=v$.
\end{rmk}

%%%%%
\section{Applications and generalisations}
%%%%%
After showing how Theorem \ref{thm:main} encompasses previous versions of the sewing lemma, we reformulate it in terms an approximate action of the pair groupoid of $\mathbb R$, and address the replacement of the real numbers by any metric space.
\subsection{Applications}
%%%

%Check that the original versions \cite{FP2006,FPM2008} as well as Bailleul's and Brault--Lejay's versions \cite{B2015,BL2018a,BL2018b} can be derived from Theorem \ref{thm:main}. Other applications?

\begin{exam}\rm
Consider the constant family $M_t=E$ where $E$ is a Banach space, and consider a local approximate flow defined by translations:
$$
	\mu_{st}(x):=x+\wt\mu(s,t)
$$
for any $x\in E$. These are isometries, hence all Lipschitz constants involved are equal to one. The local approximate flow property is expressed as

\begin{itemize}
\item $\wt\mu(t,t)=0$ for any $t\in\mathbb R$,
\item $|\wt\mu(s,u)+\wt\mu(u,t)-\wt\mu(s,t)|\le \sum_{i=1}^nC_i|t-u|^{a_i}|u-s|^{b_i}$ for any $s,t,u\in\mathbb R$,
\end{itemize}
with $a_i+b_i=1+\varepsilon$. Theorem \ref{thm:main} for this example amounts to the existence of a unique map $\wt\varphi:\mathbb R\times \mathbb R\to E$ such that $\wt\varphi(s,t)=\wt\varphi(s,u)+\wt\varphi(u,t)$ and $|\wt\varphi(s,t)-\wt\mu(s,t)|\le C|t-s|^{1+\varepsilon}$. This is the original formulation of the sewing lemma \cite{Gubi2004,FP2006}.
\end{exam}

\begin{exam}\rm
A sewing lemma for maps with values in a (possibly noncommutative) complete topological monoid first appeared in \cite{FPM2008} (Theorem 6 therein), in a more general setting than H\" older continuity, when the map $(s,t)\mapsto |t-s|$ is replaced by a more general, not necessarily continuous control map $\omega$. A more precise statement in the case $\omega(s,t)=|t-s|$ was obtained by I. Bailleul \cite[Theorem 1]{B2015}, but only in the case $\omega(s,t)=|t-s|$ and when the monoid at stake is the maps from a Banach space into itself.\\

Sticking to $\omega(s,t)=|t-s|$, consider again the case of the constant family $M_t=E$ where $E$ is a Banach space, with a local approximate flow $(\mu_{st})$ in the sense of Definition \ref{local-approx}. Applying Theorem \ref{thm:main} in this case gives back Theorem 1 in \cite{B2015}: let us point out that the Lipschitz control of Definition \ref{local-approx} is a reformulation of Hypothesis H1 in \cite[Theorem 1]{B2015}, adapted to our setting in which the metric space varies with the parameter. Encompassing  \cite[Theorem 6]{FPM2008} in our setting would require to adapt our results to general control maps $\omega$, which should be possible.
\end{exam}
 %%%%%
\subsection{Reminders on groupoids}\label{groupoids}
%%%%%

Recall that a category is called small if all of its objects and morphisms both form sets \cite{Roman2017}. A groupoid is a small category in which any morphism is invertible. In other words, a groupoid is given by a set $G$, a subset $G^{(0)}$, two surjective maps $r$ (the \textbf{range}) and $s$ (the \textbf{source}) from $G$ onto $G^{(0)}$, an involution 
\begin{eqnarray*}
	\iota:G	&\longrightarrow& G\\
	\gamma	&\longmapsto& \gamma^{-1}
\end{eqnarray*}
and a composition law
\begin{eqnarray*}
	m:G^{(2)}			&\longrightarrow &G\\
	(\gamma,\gamma')	&\longmapsto& \gamma\gamma'
\end{eqnarray*}
where
\[
	G^{(2)}:=\{(\gamma,\gamma')\in G\times G,\, s(\gamma)=r(\gamma')\},
\]
subject to the following axioms:
\begin{eqnarray*}
&&r\restr{G^{(0)}}=s\restr{G^{(0)}}=\mop{Id}_{G^{(0)}},\\
&&\gamma s(\gamma)=r(\gamma)\gamma=\gamma,\\
&&r(\gamma\gamma')=r(\gamma),\hskip 6mm s(\gamma\gamma')=s(\gamma'),\\
&&r\circ\iota=s,\hskip 6mm s\circ\iota=r,\\
&&\gamma\circ\gamma^{-1}=r(\gamma),\hskip 6mm \gamma^{-1}\circ\gamma=s(\gamma),\\
&&(\gamma\gamma')\gamma''=\gamma(\gamma'\gamma''),
\end{eqnarray*}
which are best visualised on the following diagrams:
$$
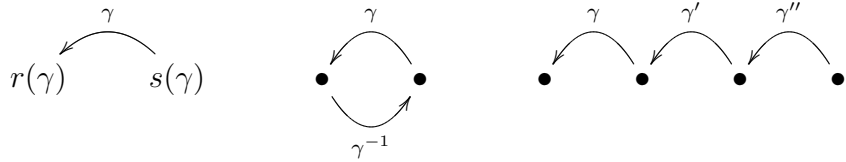
\xymatrix{
r(\gamma)&s(\gamma)\ar@/_1.5pc/[l]_{\gamma}
}
\hskip 12mm
\xymatrix{
\bullet\ar@/_1.5pc/[r]_{\gamma^{-1}}&\bullet\ar@/_1.5pc/[l]_{\gamma}
}
\hskip 12mm
\xymatrix{
\bullet&\bullet\ar@/_1.5pc/[l]_{\gamma}&\bullet\ar@/_1.5pc/[l]_{\gamma'}&\bullet\ar@/_1.5pc/[l]_{\gamma''}
}
$$
\textbf{Example 1: The pair groupoid of a set $P$.} It is defined by $G=P\times P$, with units $G^{(0)}=\{(x,x),\, x\in P\}$, source $s(x,y)=y$, range $r(x,y)=x$, inverse $(x,y)^{-1}=(y,x)$ and composition $(x,y)(y,z)=(x,z)$. The set of units is identified with $P$ itself, through the diagonal embedding $x\mapsto (x,x)$. For any groupoid $G$ with units $G^{(0)}$, the map $(r,s):G\to G^{(0)}\times G^{(0)}$ is a groupoid morphism.

\begin{defn}
A path in a topological space $P$ is a continuous map $\gamma:[0,1]\to P$. The concatenation of two paths $\gamma$ and $\gamma'$ such that $\gamma(1)=\gamma'(0)$ is defined by
\begin{eqnarray*}
	\gamma*\gamma'(t)&=&\gamma(2t) \hbox{ if } t\in [0,1/2],\\
	\gamma*\gamma'(t)&=&\gamma'(2t-1) \hbox{ if } t\in [1/2,1].
\end{eqnarray*}
The \textbf{reverse} concatenation will be denoted by a simple dot, namely
\[
	\gamma.\gamma':=\gamma'*\gamma.
\]
It makes sense whenever $\gamma(0)=\gamma'(1)$. For any path $\gamma$ and for any $s,t\in[0,1]$, the path $\gamma_{st}$ is defined by
\[
	\gamma_{st}(u):=\gamma\big(su+t(1-u)\big).
\]
\end{defn}

\noindent In particular, $\gamma_{1,0}=\gamma$. The path $\gamma_{0,1}$ is the path $\gamma$ reversed, denoted by $\gamma^{-1}$.\\

\noindent \textbf{Example 2. The fundamental groupoid of a topological space.} Recall that two paths $\gamma^0$ and $\gamma^1$ are homotopic if and only if
\begin{itemize}
\item they share the same starting and end points $a$ and $b$, i.e. $\gamma^0(0)=\gamma^1(0)=a$ and $\gamma^0(1)=\gamma^1(1)=b$,
\item there exists a homotopy between $\gamma^0$ and $\gamma^1$, i.e. a continuous map $H:[0,1]^2\to P$ such that $H(0,t)=\gamma^0(t)$ and $H(1,t)=\gamma^1(t)$ for any $t\in[0,1]$, and moreover $H(s,0)=a,\, H(s,1)=b$ for any $s\in[0,1]$.
\end{itemize}
We shall also use the convenient notation $\gamma^s(t):=H(s,t)$.\\

Homotopy is an equivalence relation denoted by $\sim$. Homotopy is compatible with concatenation: if $\gamma^0\sim\gamma^1$ and ${\gamma'}^0\sim{\gamma'}^1$, one has $\gamma^0{\gamma'}^0\sim\gamma^1{\gamma'}^1$. The fundamental groupoid of $P$, denoted by $\mathcal G_P$, is the set of homotopy classes of continuous paths in $P$ endowed with reversed concatenation. One can easily check the following chain rule:
\[	
	\gamma_{st}\sim\gamma_{su}.\gamma_{ut}.
\]
 The (reversed) concatenation is not associative, but it is well-known that
\begin{equation}\label{homot-asso}
	(\gamma''.\gamma').\gamma \sim \gamma''.(\gamma'.\gamma).
\end{equation}
The homotopy between both sides of \eqref{homot-asso} is explicitly given by
\begin{eqnarray*}
	H(s,t)
	&=&\gamma(t) \hbox{ if } 0\le t\le \frac 14(2-s),\\
	&&\gamma'(t) \hbox{ if } \frac 14(2-s)\le t\le \frac 14(3-s),\\
	&&\gamma''(t) \hbox{ if } \frac 14(3-s)\le t\le 1.
\end{eqnarray*}
\noindent\textbf{Example 3. The equivalence classes of paths in a topological space modulo thin equivalence} (\cite[Paragraph 2.1.1]{B1991}, see also \cite{CP1994, LO2023, BCFP2024}). A \textsl{loop} is a path $\gamma$ from $a$ to the same point $a$. The loop $\gamma$ is \textsl{thin} if it is homotopic to the trivial loop with image $\{a\}$, by means of a homotopy $H:[0,1]^2\to P$ such that $H([0,1]^2)\subseteq \gamma([0,1])$. It is easily seen that any loop of the form $\gamma.\gamma^{-1}$ is thin.\\

Two paths $\gamma^0$ and $\gamma^1$ with starting point $a$ and endpoint $b$ are \textsl{thin equivalent} \cite{CP1994} if there is a finite collection $\beta^0,\ldots,\beta^k$ of paths with starting point $a$ and endpoint $b$, with $\beta^0=\gamma^0$ and $\beta^k=\gamma^1$, such that the concatenation $(\beta^j)^{-1}.\beta^{j-1}$ is thin for any $j=1,\ldots, k$.\\

Thin equivalence is compatible with (reverse) concatenation. It is easily seen that a path $\gamma^0$ is thin equivalent to any of its reparametrisations $\gamma^1=\gamma^0\circ\varphi$, where $\varphi:[0,1]\to[0,1]$ is a continuous map such that $\varphi(0)=0$ and $\varphi(1)=1$: an explicit homotopy is given by $\gamma^s(t):=\gamma^0\circ\varphi_s(t)$ with $\varphi_s(t):=(1-s)t+s\varphi(t)$. The following homotopy between the constant loop $\gamma^0(0)$ and $(\gamma^0)^{-1}\gamma^1$, given by 
\begin{eqnarray*}
	K(s,t)
	&=&\gamma^s(2t) \hbox{ if } 0\le t\le \frac s2,\\
	&=& \gamma^s(s) \hbox{ if } \frac s2\le t\le 1-\frac{\varphi_s(s) }2,\\
	&=& \gamma^0(2-2t) \hbox{ if } 1-\frac {\varphi_s(s) }2\le t\le 1,
\end{eqnarray*}
has therefore its image entirely contained in $\gamma^0([0,1])$.\\

In particular,  $(\gamma''.\gamma').\gamma$ is thin equivalent to $\gamma''.(\gamma'.\gamma)$. The thin equivalence classes of paths therefore form a groupoid, denoted by $\mathcal G_P^{\smop{th}}$. It is easily seen that two thin equivalent paths are homotopic. The inclusion of thin equivalence classes into homotopy classes induces a natural surjective groupoid morphism
\[
	\pi:\mathcal G_P^{\smop{th}}\longrightarrow\hskip -9pt\to \mathcal G_P.
\]
\begin{rmk}\rm
The surjective morphism above turns out to be an isomorphism. This very surprising and counter-intuitive fact has been recently established by J. Brazas, G. R. Conner, P. Fabel and C. Kent \cite{BCFP2024} who, for any homotopy $H$ linking a pair of $P$-valued paths $(\gamma^0,\gamma^1)$, gave an explicit construction of a path $\Gamma:[0,1]\to P$ thin-homotopic to both $\gamma^0$ and $\gamma^1$. The main ingredient in their construction is the (continuous and non-decreasing) ternary Cantor map
\[\tau:[0,1]\to[0,1]\]
defined as follows: $\tau(1)=1$, and if $x\in [0,1[$ admits the expansion $0.a_1a_2a_3\cdots$ in base $3$, its image $\tau(x)$ is defined by its expansion $0.b_1b_2b_3\cdots$ in base $2$, where 
\begin{itemize}
\item $b_j=0$ if $a_j=0$ or $a_j=1$,
\item $b_j=1$ if $a_j=2$.
\end{itemize}
The image of this Peano-like path $\Gamma$ coincides with the whole image of the homotopy $H$.
\end{rmk}
\begin{rmk}\label{pair-R}\rm
The three groupoids above coincide for $P=\mathbb R$ endowed with its usual topology, i.e.
\[
	\mathcal G_{\mathbb R}^{\smop{th}}=\mathcal G_{\mathbb R}=\mathbb R\times \mathbb R.
\]
\end{rmk}

\begin{defn}
Let $G$ be a groupoid, and let $E=(E_x)_{x\in G^{(0)}}$ be a family of sets parametrised by $G^{(0)}$. A \textbf{left action of $G$ on $E$} is a family of maps $\varphi=(\varphi_\gamma)_{\gamma\in G}$ where:
\begin{itemize}
\item $\varphi_\gamma:E_{s(\gamma)}\to E_{r(\gamma)}$,
\item $\varphi_x=\mop{Id}_{E_x}$ for any $x\in G^{(0)}$,
\item $\varphi_{\gamma\gamma'}=\varphi_\gamma\circ\varphi_{\gamma'}$ for any $(\gamma,\gamma')\in G^{(2)}$.
\end{itemize}
In particular, the $\varphi_\gamma$'s are bijective. The \textbf{action of the groupoid $G$ on its set of units} is the particular case when $E_x$ is reduced to a single element for any $x\in G^{(0)}$.
\end{defn}

\begin{defn}
A \textbf{topological groupoid} is both a groupoid and a topological space, such that all structure maps are continuous, the inverse being a homeomorphism. A left action of a topological groupoid $G$ on a family $E=(E_x)_{x\in G^{(0)}}$ of topological spaces is \textbf{continuous} if all maps $\varphi_\gamma$ are continuous.
\end{defn}

%%%
\subsection{The groupoid of Lipschitz paths modulo thin Lipschitz equivalence}\label{sewing-generalized}
%%%

Let $P$ be a metric space. Recall that a path $\gamma:[0,1]\to P$ is Lipschitz if 
\[\mop{Lip}(\gamma)=\mopl{sup}_{s,t\in[0,1], s\neq t}\frac{d_P\big(\gamma(s),\gamma(t)\big)}{|s-t|}<+\infty.\]

\begin{defn}
Two Lipschitz paths $\gamma^0$ and $\gamma^1$ are \textbf{Lipschitz-homotopic} if there exists a Lipschitz homotopy $H:[0,1]^2\to P$ between both paths.
\end{defn}
\noindent Note that two Lipschitz paths can be homotopic without being Lipschitz-homotopic.
\begin{defn}
Two Lipschitz paths $\gamma^0$ and $\gamma^1$ from $a$ to $b$ are \textbf{Lipschitz-thin-equivalent} if there is a finite collection $\beta^0,\ldots,\beta^k$ of paths with starting point $a$ and endpoint $b$, with $\beta^0=\gamma^0$ and $\beta^k=\gamma^1$, such that the concatenation $(\beta^j)^{-1}.\beta^{j-1}$ is Lipschitz-thin for any $j=1,\ldots, k$, i.e. if the loop  $(\beta^j)^{-1}\beta^{j-1}$ is Lipschitz-thin, i.e. there exists a Lipschitz homotopy $H:[0,1]^2\to P$ from $(\beta^j)^{-1}\beta^{j-1}$ to the constant path $t\mapsto a$ such that the thinness condition $H([0,1]^2)\subseteq \gamma([0,1])$ holds.
\end{defn}
\noindent Lipschitz-thin equivalence, similarly to thin-equivalence, is an equivalence relation. 
\ignore{ 
\textcolor{blue}{But any thin loop $\gamma$ is Lipschitz-thin, in the sense that the homotopy from $\gamma$ to the constant loop $\{\gamma(0)\}$ can be chosen as a Lipschitz homotopy. To see that, let $H:[0,1]^2\to P$ be a homotopy from $\gamma$ to the constant loop such that $H([0,1]^2)\subset \gamma([0,1])$. Recalling the notation $\gamma^s(t)=H(s,t)$, we choose the Lipschitz homotopy as follows:
\begin{eqnarray*}
	\overline H(s,t)
	&=& \gamma (t) \hbox{  if }\gamma (t)\in \gamma^s([0,1]),\\
	&=&\gamma(t_s) \hbox{ if not},
\end{eqnarray*}
with $t_s:=\mop{inf}\{t\in[0,1],\, \gamma(t)\notin \gamma^s([0,1])\}$. As a consequence, thin equivalence makes perfect sense in the subclass of Lipschitz paths.}
}
\begin{prop}
For any pair of paths $(\gamma,\gamma')$ in a metric space $P$, the following estimate holds:
\[
	\mop{Lip}(\gamma.\gamma')\le 2\mop{sup}\big(\mop{Lip}(\gamma),\,\mop{Lip}(\gamma')\big).
\]
Moreover, for any $s,t\in[0,1]$ we have
\[
	\mop{Lip}(\gamma_{st})\le |t-s|\mop{Lip}(\gamma).
\]
\end{prop}

\begin{proof}
This is immediate. The factors $2$ and $|t-s|$ come from rescaling.
\end{proof}

Thin Lipschitz equivalence is compatible with concatenation, and $\gamma.(\gamma'.\gamma'')$ is Lipschitz-thin equivalent to $(\gamma.\gamma').\gamma''$ for any triple of Lipschitz paths composable that way. This therefore defines the \textbf{Lipschitz-thin equivalence groupoid} $\mathcal G_P^{\smop{Lip-th}}\subseteq \mathcal G_P^{\smop{th}}$ as the equivalence classes of Lipschitz paths modulo thin equivalence. The inclusion may be strict, because there can exist a path in $P$ which is not thin-homotopic to any Lipschitz path: an example of this situation is given by the image of a non-Lipschitz path (e.g. a sample path of a Brownian motion) in $\mathbb R^2$.
%%%
\subsection{The Lipschitz homotopy groupoid}
%%%
This notion is adapted from \cite[Definition 4.1]{DHLT}. Let $P$ be a metric space, and let $x,y\in P$. Two Lipschitz paths with domain $[0,1]$ from $x$ to $y$ are Lipschitz-homotopic if there exists a Lipschitz homotopy $H:[0,1]^2\to P$. This notion is obviously compatible with concatenation. The Lipschitz homotopy groupoid
$\mathcal G_P^{\smop{Lip}}$ is the set of Lipschitz paths $\gamma:[0,1]\to P$ modulo Lipschitz homotopy.\\

Both groupoids $\mathcal G_P$ and $\mathcal G_P^{\smop{Lip}}$ often coincide, for example when $P$ is a smooth manifold: in this case, this is due to the fact that any continuous path is homotopic to a smooth (hence Lipschitz) one, and any two homotopic smooth paths can be interpolated by a smooth homotopy.

%%%
\subsection{A conjectural generalisation of Theorem \ref{thm:main}}\label{sect:gen-sewing}
%%%

Approximate local flows, in the sense of Definition \ref{local-approx}, can be straightforwardly generalised as follows:

\begin{defn}\label{local-approx-gen}
Let $P$ be a metric space, and let $(M_x,d_x)_{x\in P}$ be a family of complete metric spaces indexed by $P$, possibly in the extended sense. An \textbf{approximate action of the pair groupoid $P\times P$} associated with these data is a two-parameter family $(\mu_{xy})_{x,y\in P}$ of maps such that
\begin{itemize}
\item $\mu_{xy}\in C_{xy}$,
\item $\mu_{xx}=\mop{Id}_{M_x}$ for any $x\in P$,
\item There exists a continuous function $f:\mathbb R\to[0,+\infty[$ with $f(0)=0$, such that 
\[\mop{Lip}(\mu_{xy})\le 1+f\big(d_P(x,y)\big),\]
\item There exists $\varepsilon>0$, an integer $n\ge 1$ and two collections of positive constants $a_i,b_i, i=1,\ldots,n$ with $a_i+b_i=1+\varepsilon$, such that for any $x,y,z\in P$ the following estimates hold:
\begin{equation}\label{laf-ciP}
d_{xy}\big(\mu_{xy},\mu_{xz}\circ\mu_{zy}\big)\le \sum_{i=1}^n C_id_P(y,z)^{a_i}d_P(z,x)^{b_i}.
\end{equation}
\item There exists a non-decreasing continuous function $g:[0,+\infty[\to [1,+\infty[$ such that, for any $x,y\in P$ and any ordered collection $I=(x_0,\ldots,x_k)$ of elements of $P$ with $x_0=x$ and $x_k=y$, the map $\mu_{xy}^I:=\mu_{x_0x_1}\circ\mu_{x_1x_2}\circ\cdots\circ\mu_{x_{k-1}x_k}$ satisfies
\[\mop{Lip}\mu_{xy}^I\le g\left(\sum_{j=1}^k d_P(x_j,x_{j-1})\right).\]
\end{itemize}
\end{defn}
\begin{rmk}
Estimates \eqref{laf-ciP} are equivalent to the four-point estimates
\begin{equation}\label{ineq-xuvy}
d_{xy}(\mu_{xu}\circ\mu_{uy},\, \mu_{xv}\circ\mu_{vy})\le \big(1+Ld_P(u,x)\big)\sum_{i=1}^n C_id_P(y,v)^{a_i}d_P(u,v)^{b_i}+\sum_{i=1}^nC_i d_P(x,u)^{b_i}d_P(u,v)^{a_i}.
\end{equation}
the same way \eqref{laf-ci} are equivalent to \eqref{ineq-suvt} (see Remark \ref{four-point}).
\end{rmk}
\begin{con}\label{thm:metric}
Let $P$ be a metric space, and let  $\big((M_x),\,(\mu_{xy})\big)$ be an approximate action of the pair groupoid $P\times P$ in the sense of Definition \ref{local-approx-gen}. There exists a unique action $\varphi$ of the Lipschitz thin equivalence groupoid $\mathcal G_P^{\smop{Lip-th}}$ such that, for any representative path $\gamma:[0,1]\to P$ of a given class $\mathbf g\in \mathcal G_P^{\smop{Lip-th}}$,
\begin{equation}\label{estimate-phi}
d_{\gamma(s)\gamma(t)}(\varphi_{[\gamma_{st}]},\,\mu_{\gamma(s)\gamma(t)})\le C'|t-s|^{1+\varepsilon},
\end{equation}
where $\gamma_{st}$ is the class of the path $\gamma_{st}$ in $\mathcal G_P^{\smop{Lip-th}}$. The map $\varphi_{\mathbf g}$ is obtained, for any choice of $\gamma\in\mathbf g$, as the limit in $C_{\gamma(1)\gamma(0)}$ of the compositions
$$
	\mu_{\gamma(0)\gamma(t_1)}\circ\mu_{\gamma(t_1)\gamma(t_2)}\circ\cdots\circ\mu_{\gamma(t_{k-1})\gamma(1)}
$$
when the mesh of the subdivision $(0,t_1,\ldots, t_{k-1},1)$ of the interval $[0,1]$ tends to zero. One possible constant $C'$ is given by
$$
	C'=2^{1+\varepsilon}g\big(\mop{Lip}(\gamma)|t-s|\big)\left(\sum_{i=1}^nC_i\right)
	\zeta(1+\varepsilon).
$$
Moreover, the action $\varphi$ satisfies the Lipschitz condition:
\begin{equation}
\mop{Lip}(\varphi_{\mathbf g})\le g\Big(\mopl{inf}_{\gamma\in\mathbf g}\mop{Lip}(\gamma)\Big).
\end{equation}
\end{con}

\begin{proof}[Tentative proof]
 Theorem \ref{thm:main} applied to the local approximate flow $(\mu^\gamma_{st}):=(\mu_{\gamma(s)\gamma(t)})$ defines an action of the up-to-homotopy groupoid of paths, that is to say a map $\gamma\mapsto \psi_\gamma\in C_{\gamma_0\gamma_1}$ defined on continuous paths such that $\psi_{\gamma.\gamma'}=\psi_{\gamma}\circ\psi_{\gamma'}$. Let us remark that, as a consequence of Theorem \ref{thm:main}, the map $\gamma\mapsto\psi_{\gamma}$ is invariant under reparameterisation, and $\psi(\gamma^{-1})=\psi(\gamma)^{-1}$. The problem resides in building the action $\varphi$ from $\psi$, by showing that $\psi_\gamma$ only depends on the Lipschitz thin homotopy class of the path $\gamma$. We think that it would be possible in the context of $C^1$ paths with values in a $C^1$ manifold, using the Tlas decomposition of a thin $C^1$ loop \cite[Theorem 2]{Tlas2016}. We haven't been able to find an efficient substitute of Tlas' decomposition of for a general Lipschitz thin path so far.
\end{proof}
\subsection{The knitting lemma}\label{sect:knitting}
The tempting comparison with the higher-dimensional sewing lemma of \cite{H2021} (see also \cite{LO2023}, \cite{L2024} and \cite{DEHT2024} in the two-dimensional case) is not straightforward. One however may look for a stronger notion of approximate action of the pair groupoid $P\times P$ so that the action $\gamma\mapsto\varphi_\gamma$ given by Conjecture \ref{thm:metric} factorises through an action of the fundamental groupoid $\mathcal G_P$, or at least the Lipschitz homotopy groupoid $\mathcal G_P^{\smop{Lip}}$. A possible answer is given by the following theorem:
\begin{thm}[knitting lemma]\label{thm:knitting-lemma}
Let $P$ be a metric space. Let  $\big((M_x),\,(\mu_{xy})\big)$ be a strongly approximate action of the pair groupoid $P\times P$, in the sense that the three-point estimate \eqref{laf-ciP} in Definition \ref{local-approx-gen} (or, equivalently, the four-point estimate \eqref{ineq-xuvy}) is replaced by the stronger four-point estimate
\begin{equation}\label{ineq-xuvy-strong}
d_{xy}(\mu_{xu}\circ\mu_{uy},\, \mu_{xv}\circ\mu_{vy})\le \big(1+Ld_P(u,x)\big)\sum_{i=1}^n C_id_P(y,v)^{a_i}d_P(u,v)^{b_i}+\sum_{i=1}^nC_i d_P(x,u)^{b_i}d_P(u,v)^{a_i}.
\end{equation}
for any $x,u,v,y\in P$, where $(a_i,b_i)_{i=1,\ldots,n}$ is a collection of positive numbers such that
\[a_i+b_i=2+\varepsilon\]
with $\varepsilon>0$. Then there exists a unique continuous action $\varphi$ of the Lipschitz fundamental groupoid $\mathcal G_P^{\smop{Lip}}$ such that
\begin{equation}\label{estimate-phi}
d_{\gamma(s)\gamma(t)}(\varphi_{\gamma_{st}},\,\mu_{\gamma(s)\gamma(t)})\le C'|t-s|^{1+\varepsilon},
\end{equation}
which is obtained as the limit in $C_{\gamma(s)\gamma(t)}$ of the compositions
$$
	\mu_{\gamma(s)\gamma(t_1)}\circ\mu_{\gamma(t_1)\gamma(t_2)}\circ\cdots\circ\mu_{\gamma(t_{k-1})\gamma(t)}
$$
when the mesh of the subdivision $(s,t_1,\ldots, t_{k-1},t)$ of the interval between $s$ and $t$ tends to zero. One possible constant $C'$ is given by
\begin{equation}\label{C'-strong}
	C'=2^{1+\varepsilon}e^{L\smop{Lip}(\gamma)|t-s|}\left(\sum_{i=1}^nC_i\right)
	\zeta(2+\varepsilon).
\end{equation}
Moreover, the continuous action $\varphi$ satisfies the Lipschitz condition:
\begin{equation}
\mop{Lip}(\varphi_{\gamma_{st}})\le e^{L\smop{Lip}(\gamma)|t-s|}.
\end{equation}
\end{thm}
\begin{proof}
The proof of Theorem \ref{thm:knitting-lemma} requires two steps.\\

\noindent{\bf Step 1:}
For any Lipschitz path $\gamma:[0,1]\to P$, Equations \eqref{estimate-phi} and \eqref{C'-strong} are obtained from Theorem \ref{thm:main} applied to the approximate flow $\big((\mu_\gamma)_{s,t}\big)_{s,t\in[0,1]}$ given by $(\mu_\gamma)_{s,t}:=\mu_{\gamma(s)\gamma(t)}$, with the modification $\zeta(2+\varepsilon)$ instead of $\zeta(1+\varepsilon)$ coming from the strong four-point estimate \eqref{ineq-xuvy-strong}. The Lipschitz parameter from Definition \ref{local-approx} for the approximate flow $(\mu_\gamma)$ is indeed $L\mop{Lip}(\gamma)$. Details are left to the reader.\\

\noindent{\bf Step 2:}
It remains to show that $\varphi_\gamma$ does not depend on the choice of the path $\gamma$ in a given Lipschitz homotopy class. Let $x,y\in P$, let $\gamma^0$ and $\gamma^1$ be two Lipschitz continuous paths from $[0,1]$ to $P$ with starting point $x$ and endpoint $y$. Let $H:[0,1]^2\to P$ be a Lipschitz homotopy from $\gamma^0$ to $\gamma^1$. We recall the convenient notation $\gamma^s(t):=H(s,t)$. Let $\ell$ be the Lipschitz norm of $H$, namely
\[\ell:=\mopl{sup}_{t,s,t',s'\in[0,1], (s,t)\neq(s',t')} \frac{d_P\big(\gamma^s(t),\gamma^{s'}(t')\big)}{|t'-t|+|s'-s|}.\]
In particular, the Lipschitz norm of any path $\gamma^s,\, s\in[0,1]$ is smaller or equal to $\ell$. As a consequence, $d_P(x,y)\le \ell$.\\

Let $k\ge 2$ be an integer, let $0=t_0\le t_1\cdots\le t_k=t$ be the regular subdivision of $[0,1]$ with mesh $\delta=1/k$, and let us denote $H(t_i,t_j)=\gamma^{t_i}(t_j)$ by $x_j^{i}$ for any $i,j\in\{0,\ldots, k\}$. The \textsl{mesh of the net $I=(x_j^i)_{i,j=0,\ldots, k}$} is defined by
\begin{equation}
\mop{mesh}_H(I):=\mop{sup}\left\{\mop{sup}_{i=0,\ldots, k,\, j=1,\ldots, k}d_P(x^i_{j-1},x^i_j),\, \mop{sup}_{i=1,\ldots, k,\, j=0,\ldots, k}d_P(x^{i-1}_{j},x^i_j)\right\}.
\end{equation}
By definition we have $x_0^i=x$ and $x_k^i=y$ for any $i\in\{0,\ldots ,k\}$. The following estimate for the mesh of the net clearly holds true:
\begin{equation}\label{mesh-knit}
\mop{mesh}_H(I)\le \delta\ell.
\end{equation}
\vskip 4mm
\tikzset{every picture/.style={line width=0.75pt}} %set default line width to 0.75pt        
\hskip 45mm
\begin{tikzpicture}[x=0.75pt,y=0.75pt,yscale=-1,xscale=1]
%uncomment if require: \path (0,557); %set diagram left start at 0, and has height of 557

%Curve Lines [id:da8270096514430112] 
\draw    (179,349.9) .. controls (205.5,267.9) and (362.5,256.9) .. (421.5,346.9) ;
%Curve Lines [id:da7848617103666157] 
\draw    (179,349.9) .. controls (202.5,323.9) and (372.5,318.9) .. (421.5,346.9) ;
%Curve Lines [id:da6077557982876254] 
\draw    (179,349.9) .. controls (218.5,343.9) and (385,352.9) .. (421.5,346.9) ;
%Curve Lines [id:da10383773156005727] 
\draw    (179,349.9) .. controls (223.5,366.9) and (381.5,376.9) .. (421.5,346.9) ;
%Curve Lines [id:da5804367739646856] 
\draw    (179,349.9) .. controls (231.5,477.9) and (419.5,416.9) .. (421.5,346.9) ;
%Curve Lines [id:da4965358343408981] 
\draw    (179,349.9) .. controls (197.5,306.9) and (368.5,289.9) .. (421.5,346.9) ;
%Curve Lines [id:da3439568488508413] 
\draw    (179,349.9) .. controls (226.5,414.9) and (380.5,408.9) .. (421.5,346.9) ;
%Curve Lines [id:da19511587649129747] 
\draw    (293.5,424.5) .. controls (333.5,393.9) and (251.5,314.9) .. (291.5,284) ;
%Curve Lines [id:da28137417492038896] 
\draw    (251.5,289) .. controls (239.5,338.9) and (272.5,387.9) .. (268.5,421.9) ;
%Curve Lines [id:da9793310162160342] 
\draw    (200.5,316.9) .. controls (220.5,346.9) and (209.5,370.9) .. (212.5,397) ;
%Curve Lines [id:da34888678953512264] 
\draw    (333.5,420) .. controls (357.5,393.25) and (308.5,322.1) .. (348.5,292.1) ;
%Curve Lines [id:da37057326362817655] 
\draw    (369.5,407.25) .. controls (353.5,374.25) and (342.5,338.25) .. (382.5,308.25) ;

% Text Node
\draw (156,339.4) node [anchor=north west][inner sep=0.75pt]    {$x$};
% Text Node
\draw (429,335.4) node [anchor=north west][inner sep=0.75pt]    {$y$};
% Text Node
\draw (183,287.4) node [anchor=north west][inner sep=0.75pt]    {$x_{1}^{0}$};
% Text Node
\draw (236,263.4) node [anchor=north west][inner sep=0.75pt]    {$ \begin{array}{l}
x_{2}^{0}\\
\end{array}$};
% Text Node
\draw (284,256.4) node [anchor=north west][inner sep=0.75pt]    {$x_{3}^{0}$};
% Text Node
\draw (342,265.4) node [anchor=north west][inner sep=0.75pt]    {$x_{4}^{0}$};
% Text Node
\draw (380,284.4) node [anchor=north west][inner sep=0.75pt]    {$x_{5}^{0}$};
% Text Node
\draw (192,398.4) node [anchor=north west][inner sep=0.75pt]    {$x_{1}^{6}$};
% Text Node
\draw (254,425.4) node [anchor=north west][inner sep=0.75pt]    {$x_{2}^{6}$};
% Text Node
\draw (287,430.4) node [anchor=north west][inner sep=0.75pt]    {$x_{3}^{6}$};
% Text Node
\draw (329,422.4) node [anchor=north west][inner sep=0.75pt]    {$x_{4}^{6}$};
% Text Node
\draw (366,413.4) node [anchor=north west][inner sep=0.75pt]    {$x_{5}^{6}$};
% Text Node
\draw (303,367.4) node [anchor=north west][inner sep=0.75pt]    {$x_{3}^{4}$};
% Text Node
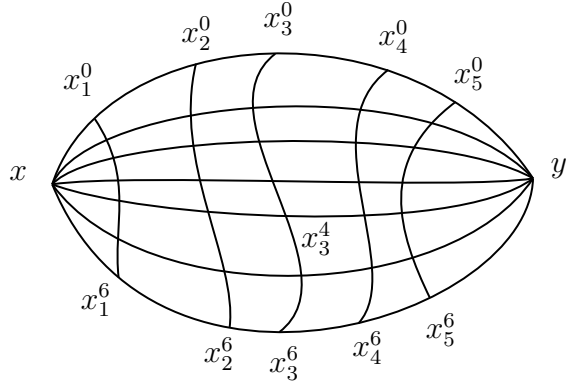
\draw (196,471.4) node [anchor=north west][inner sep=0.75pt]    {\small An\ example\ of\ a\ net\ with\ $k=6$.};

\end{tikzpicture}%%
\vskip 4mm
Let $(\mu_{uv})$ be a strongly approximate action of the pair groupoid $P\times P$. Let us consider for any $i,j\in\{0,\ldots,k-1\}$:
\begin{equation}\label{tricot}
\mu_{xy}^{I,ij}:=\mu_{x_0^ix_1^i}\circ\mu_{x_1^ix_2^i}\circ\cdots\circ\mu_{x_{k-j-1}^ix_{k-j}^{i+1}}\circ\cdots\circ\mu_{x_{k-1}^{i+1}x_k^{i+1}}.
\end{equation}
In particular,
\begin{equation}\label{rang}
\mu_{xy}^{I,(i-1)(k-1)}=\mu_{xy}^{I,i0}=\mu_{x_0^{i}x_1^{i}}\circ\cdots\circ\mu_{x_{k-1}^ix_k^i}.
\end{equation}
We are now in position to compute:
\begin{eqnarray*}
d_{xy}&&\hskip -10mm(\mu_{x_0^{0}x_1^{0}}\circ\cdots\circ\mu_{x_{k-1}^0x_k^0},\, \mu_{x_0^{k}x_1^{k}}\circ\cdots\circ\mu_{x_{k-1}^kx_k^k})\\
&\le&\sum_{i\in\{0,\ldots,k-1\}}\left(\sum_{j\in\{0,\ldots,k-2\}}d_{xy}(\mu_{xy}^{I,ij},\, \mu_{xy}^{I, i(j+1)})\right)+d(\mu_{xy}^{I, i(k-1)},\, \mu_{xy}^{I, (i+1)0})\\
&\le&\sum_{i\in\{0,\ldots,k-1\}}\sum_{j\in\{0,\ldots,k-2\}}d_{xy}(\mu_{xy}^{I,ij},\, \mu_{xy}^{I, i(j+1)}) \hskip 8mm \hbox{(from \eqref{rang})}\\
&\le&k^2 e^{\delta \ell L}(2+\delta\ell L)\left(\sum_{r=1}^n C_i\right)(\delta\ell)^{2+\varepsilon}\hskip 8mm\hbox{(from \eqref{ineq-xuvy-strong} and \eqref{mesh-knit})}\\
&\le&e^{\delta \ell L}(2+\delta\ell L)\left(\sum_{r=1}^n C_i\right)\ell^{2+\varepsilon}\delta^\varepsilon.
\end{eqnarray*}
The right-hand side converges to zero as $\delta\to 0$ (in other words, for $k\to+\infty$), so that we get
\[\varphi_{\gamma^0}=\varphi_{\gamma^1},\]
which concludes the proof of Theorem \ref{thm:knitting-lemma}.
\end{proof}
\begin{rmk}\rm
Proving Conjecture \ref{thm:metric} would lead to a consistent definition of holonomy for based Lipschitz loops. The knitting lemma is a flatness result in the sense that it stipulates that the holonomy of any contractible based Lipschitz loop is trivial.
\end{rmk}
%%%%%%%%%%%%
\ignore{
%%%%%
\section{Appendix: the Tlas decomposition of a Lipschitz curve}
%%%%%
\begin{thm}\label{tlas}
Let $\gamma:[0,1]\to P$ be a Lipschitz path, where $P$ is a length space. There exists a collection $(A_0,A_1,\ldots,A_n,\ldots)$ of mutually disjoint subsets of $[0,1]$ such that
\begin{itemize}
\item $A_0$ is closed, $A_n$ is open whenever $n\ge 1$, and $[0,1]=\bigsqcup_{n\ge 0} A_n$.
\item $\gamma^{-1}\big(\gamma(A_n)\big)=A_n$ for any $n\ge 0$, and, for any $t\in A_n$ with $n\ge 1$, the cardinality of $\gamma^{-1}\big(\{\gamma(t)\}\big)$ is equal to $n$.
\item The path $\gamma$ is constant on any connected component of $A_0$.
\item For any $n\ge 1$, the restriction of $\gamma$ to any connected component of $A_n$ is an embedding. Any two such embeddings are identical modulo reparameterisation and change of orientation.
\end{itemize}
\end{thm}
\begin{proof}
This is directly borrowed from Theorem 1 in \cite{Tlas2016} for $C^1$ curves, adapting T. Tlas' proof to the general Lipschitz case. The notion of critical point must be replaced by the notion of pseudo-critical point defined below. The set of pseudo-critical points shares sufficiently many good properties with the set of critical points of a $C^1$ curve to allow Tlas' proof to be adapted. For any integer $n\ge 1$, let $E_n$ be the set of $t\in[0,1]$ such that $\gamma^{-1}\big(\{\gamma(t)\}\big)$ is of cardinality $n$. and let $E_\infty$ be the set of $t\in[0,1]$ such that $\gamma^{-1}\big(\{\gamma(t)\}\big)$ is infinite. For later use, we also introduce for any integer $n\ge 1$
\[E_{\le n}:=\bigsqcup_{k=1}^n E_k,\hskip 12mm E_{\ge n+1}:=\bigsqcup_{k\in \mathbb N\sqcup\{+\infty\},\,k\ge n+1} E_k.\]
\begin{defn}
Let $\gamma:[0,1]\to P$ a continuous path, where $P$ is a length space. The point $t\in[0,1]$ is \textbf{pseudo-critical} if there exist two sequences $(u_k)_{k\ge 1}$ and $(v_k)_{k\ge 1}$ both converging to $t$, such that $u_k\neq v_k$ and $\gamma(u_k)=\gamma(v_k)$ for any $k\ge 1$. The set of pseudo-critical points will be denoted by $C$, and we set $C':=\gamma^{-1}\big(\gamma(C)\big)$.
\end{defn}
\begin{prop}\label{CC'}
$C$ and $C'$ are closed subsets of $[0,1]$, and the inclusion $E_\infty\subset C'$ holds.
\end{prop}
\begin{proof}
Let $(t_n)_{n\ge 1}$ be a sequence in $C$, converging to $t\in[0,1]$. For any $n\ge 1$, there exist two sequences $(u^n_k)_{k\ge 1},\ (v^n_k)_{k\ge 1}$ both converging to $t_n$, with $\gamma(u_n^k)=\gamma(v_n^k)$, such that $u_n^k\neq v_n^k$. There is an increasing map $k\mapsto n(k)$ such that the two "diagonal" sequences $(u_k^{n(k)})$ and $(v_k^{n(k)})$ both converge to $t$, with $u_k^{n(k)}\neq v_k^{n(k)}$ and $\gamma(u_k^{n(k)})=\gamma(v_k^{n(k)})$. This proves that $C$ is closed. As $C$ is compact (as a closed subset of the compact space $[0,1]$), the image $\gamma(C)$ is compact hence closed, hence $C'=\gamma^{-1}\big(\gamma(C)\big)$ is closed. Finally, if $t\in E_\infty$, one can extract from the infinite set $\gamma^{-1}\big(\{\gamma(t)\}\big)$ a sequence $(u_k)_{k\ge 1}$ converging to some $u\in\gamma^{-1}\big(\{\gamma(t)\}\big)$. As all $u_k$'s are pairwise disjoint and $\gamma(t_k)=\gamma(t)$, we can set $v_k:=u_{k+1}$, showing that $u$ is pseudo-critical. This proves that $C'$ contains $E_\infty$.
\end{proof}
\noindent We now set
\[B_n:=([0,1]\setminus C')\cap E_n,\hskip 8mm B_{\le n}:=([0,1]\setminus C')\cap E_{\le n},\hskip 8mm B_{\ge n+1}:=([0,1]\setminus C')\cap E_{\ge n+1}, \]
and $A_n:=\mopl{$B_n$}^\circ$. We have $\bigsqcup_{n\ge 1}B_n=[0,1]\setminus C'$, and $B_\infty=\emptyset$ according to Proposition \ref{CC'}. Setting $A_0:=[0,1]\setminus\bigsqcup_{n\ge 1}A_n$, we have
\begin{eqnarray*}
A_0&=&\left([0,1]\setminus \bigsqcup_{n\ge 1} B_n\right)\sqcup\bigsqcup_{n\ge 1}B_n\setminus A_n \\
&=&C'\sqcup\bigsqcup_{n\ge 1}B_n\setminus A_n.
\end{eqnarray*}
\noindent Let us state and prove a lemma before proving Theorem \ref{tlas}:
\begin{lem}\label{aux}
For any positive integer $m$ and for any $s\in B_m$, there exists $\delta>0$ such that
\[]s-\delta,s+\delta[ \cap\left(C'\sqcup B_{\ge m+1}\right)=\emptyset.\]
\end{lem}
\begin{proof}
Adapted almost verbatim from \cite{Tlas2016}. Assume the converse of the statement, namely that there exists a sequence $(s_k)_{k\ge 1}$ in $C'\sqcup B_{\ge m+1}$ converging to $s$. As $s\notin C'$ and $C'$ is closed, $s_k\in B_{\ge m+1}$ for any large enough $k$, hence we can suppose it is the case for all terms of the sequence, by neglecting the first terms. Hence $\gamma^{-1}\big(\{\gamma(s_k)\}\big)$ has cardinality at least $m+1$. For every $k$, let us select $m+1$ elements in $\gamma^{-1}\big(\{\gamma(s_k)\}\big)$, which gives a sequence $(s^1_k,\ldots,s^{m+1}_k)_{k\ge 1}$ in the compact space $[0,1]^{m+1}$, made of pairwise distinct elements of $[0,1]$. A subsequence of it converges to a limit $\ell=(\ell_1,\ldots,\ell_{m+1})\in [0,1]^{m+1}$. We have $\gamma(\ell_1)=\cdots =\gamma(\ell_{m+1})=\gamma(s)$. Supposing $\ell_i=\ell_j$ for some pair $(i,j)$ with $i\neq j$, both sequences $(s^i_k)_k$ and $(s^j_k)_k$ converge to $\ell_i$, and verify $s^i_k\neq s^j_k$ as well as $\gamma(s_k^i)=\gamma(s_k^j)$. Hence $\ell_i\in C$, yielding $s\in C'$, in contradiction with $s\in B_m$. We therefore have that the limits $\ell_1,\ldots,\ell_{m+1}$ are pairwise different, hence $s\in E_{\ge m+1}$, contradicting $s\in B_m$ again.
\end{proof}
The collection $(A_0,A_1,\ldots)$ clearly verifies items 1 and 2 of Theorem \ref{tlas}. Let us now prove item 3. Let $]a,b[\subset \mathop{A_0}\limits^\circ$ be an open interval. As $C'$ is closed, either $]a,b[$ is entirely included in $C'$, or there is an open subinterval $]a',b'[$ included in $\bigsqcup_{n\ge 1}B_n\setminus A_n$. The second option is impossible: in fact, let $j:=\mop{min}\{m,\, ]a',b'[\cap B_m\neq\emptyset\}$. Lemma \ref{aux} shows that $B_{\le j}=B_1\sqcup\cdots\sqcup B_j$ is open in $[0,1]$. An open subinterval $]a'',b''[$ of $]a',b'[$ is therefore contained in $\mathop{B_j}\limits^\circ=A_j$, contradicting the fact that it must be contained in $A_0$, because $A_0\cap A_n=\emptyset$ for any $n\ge 1$. So any open interval in $A_0$ is contained in $C'$.
\begin{lem}\label{E-infty}
Suppose that $]a,b[$ is included in $C'$. Then $]a,b[ \subseteq E_\infty$.
\end{lem}
\begin{proof}
From Lemma \ref{aux}, $B_{\le m}$ is open, hence its complementary $B_{\ge m+1}\cup C'=E_{\ge m+1}\cup C'$ is closed. \textcolor{red}{!!!} From $C'$ closed we get $E_{\ge m+1}\cap C'$ closed. Let us also remark that $E_\infty\cap C'=\bigcap_{m\ge 1}(E_{\ge m+1}\cap C')$ is closed as well. Choosing an integer $m\ge 1$, the set $E_{\le m}\cap C'$ is either empty or contains an open subinterval $]a',b'[$ of $]a,b[$. In the latter situation, let
\[m_0:=\mop{min}\{n\le m, E_n\cap C' \hbox{ contains an open subinterval }]a',b'[ \hbox{ of }]a,b[\}.\]
The case $m_0=1$ is impossible. In fact, $C'\cap E_1=C\cap E_1$, and if $t\in ]a',b'[\subset C\cap E_1$, there exist $t'$ and $t''$, in $]a',b'[$, distinct and arbitrarily close to $t$, such that $\gamma(t')=\gamma(t'')$. Hence $t',t''\in E_2$, contradicting the inclusion $ ]a',b'[\subset C\cap E_1$. The case $2\le m_0<+\infty$ is impossible as well. Indeed, let $[a'',b'']$ be a closed subinterval of the interval $]a',b'[\subset C'\cap E_{m_0}$, and let
\[\mu:=\mopl{min}_{t\in[a'',b'']}\,\mopl{min}_{t',t''\in\gamma^{-1}(\{\gamma(t)\})}|t'-t''|.\]
We have obviously $\mu>0$. Now let $t\in[a'',b'']$, and let $(t_k)$ be a sequence in $[a'',b'']$ converging to $t$. Let $\gamma^{-1}\big(\gamma(\{t_k\})\big)=\{t_k^1,\ldots,t_k^{m_0}\}$, with $t_k^1<\cdots <t_k^{m_0}$ and, similarly, let $\gamma^{-1}\big(\gamma(\{t\})\big)=\{t^1,\ldots,t^{m_0}\}$, with $t^1<\cdots <t^{m_0}$. The set $\{(t_k^1,\ldots,t_k^{m_0}),\,k\ge 1\}$ in $[0,1]^{m_0}$ admits a limit point $(\tau^1,\ldots,\tau^{m_0})$, with $\gamma(\tau^1)=\cdots =\gamma(\tau^{m_0})$. Hence $(\tau^1,\ldots,\tau^{m_0})=(t^1,\ldots,t^{m_0})$ is the unique limit point, which is the limit of the sequence $(t_k^1,\ldots,t_k^{m_0})_k$ in $[0,1]^{m_0}$. We therefore have shown that the map 
\begin{eqnarray*}
\varphi: [a'',b'']&\longrightarrow & [0,1]^{m_0}\\
t&\longmapsto & (t^1,\ldots,t^{m_0})
\end{eqnarray*}
is injective and continuous. Each coordinate $\varphi^j:[a'',b'']\to [0,1]$ is also injective and $\gamma^{-1}\big(\gamma([a'',b''])\big)$ is the disjoint union of the closed intervals $\varphi^i([a'',b''])$, otherwise $[a'',b'']$ would intersect $E_{m_0+1}$ nontrivially. The components  $\varphi^i([a'',b'']),\, j=1,\ldots, m_0$ are arranged in increasing order inside $[0,1]$. Now choose $t\in]a'',b''[$, choose $j\in\{1,\ldots,m_0\}$ such that $\varphi^j(t)\in C$ (such a $j$ exists by definition of $C'$), and choose $\alpha,\beta\in \varphi^i([a'',b''])$ such that $|\alpha-\beta]<\mu$ and $\gamma(\alpha)=\gamma(\beta)$, which is possible by pseudo-criticality of $\varphi^j(t)$. This contradicts the definition of $\mu$ and ends up the proof of Lemma \ref{E-infty}.
\end{proof}
We are now in position to end up proving item 3 of Theorem \ref{tlas}: if $]a,b[\subset E_\infty$ and $\gamma(a)\neq\gamma(b)$, choose a closed subinterval $[a',b']$ such that $\gamma(a')\neq\gamma(b')$. The continuous curve $\gamma$ crosses the distance in $P$ between $\gamma(a')$ and $\gamma(b')$ infinitely many times, which contradicts the fact that $\gamma$ is Lipschitz. Hence, for any Lipschitz curve $\gamma:[0,1]\to P$ and for any open interval $]a,b[\subset \mathop{A_0}\limits^\circ$, the equality $\gamma(a)=\gamma(b)$ holds.\\

\end{proof}
}
%
%%%
\ignore{\begin{enumerate}
\item \textcolor{blue}{(added 06.04.2022)} Replace the time space $\mathbb R$ by a any metric space $P$. Replace $|t-s|$ by $d_P(s,t)$ everywhere. This will be possible if $P$ verifies an additional property: \textcolor{blue}{(added 08.12.2022) Definitely cannot work. But could work if the metric space $P$ is {\bf $\mathbf 0$-hyperbolic} in the sense of Gromov.}
\begin{defn}
Let $\gamma\in ]0,1]$. A metric space $P$ is $\gamma$-regular if for any $s,t\in P$ with $s\neq t$, there exists $u\in P$ such that
\begin{equation}
d_P(u,s)\le 2^{-\gamma}d_P(s,t) \hbox{ and }d_P(u,t)\le 2^{-\gamma}d_P(s,t).
\end{equation}
\end{defn}
For example, for any positive integer $d$, any convex subset of $\mathbb R^d$ is obviously $1$-regular, as we can see by choosing the middle point $u$ on the straight line between $s$ and $t$.
\begin{prop}\cite[Proposition 1]{Gubi2004}\label{prop:fdpsewing}
Let $\gamma\in]0,1]$, let $P$ be a $\gamma$-regular metric space, let $\mu$ be a continuous function on $P\times P$ with values in a Banach space $B$, and let $\varepsilon>(1-\gamma)/\gamma$. Suppose that there exist a positive integer $n$ and two collections $a_i,b_i\ge 0$ indexed by $i\in\{1,\ldots,n\}$, with $a_i+b_i=1+\varepsilon$, such that $\mu$ satisfies:
\begin{equation}
\label{fdpmu}
	\vert\mu(s,t)-\mu(s,u)-\mu(u,t)\vert\le \sum_{i=1}^nC_id_P(u,t)^{a_i}d_P(u,s)^{b_i}
\end{equation}
for any $s,t,u\in P$ with $d_P(u,s)\le 2^{-\gamma}d_P(s,t)$ and $d_P(u,t)\le 2^{-\gamma}d_P(s,t)$, where the $C_i$'s are positive constants. Then there exists a function $\varphi \colon P \to B$, unique up to an additive constant, such that:
\begin{equation}\label{fdp-sewing}
	\vert\varphi(t)-\varphi(s)-\mu(s,t)\vert\le C'd_P(s,t)^{1+\varepsilon},
\end{equation}
where the best constant in \eqref{fdp-sewing} is
$$	
	C':=\frac{1}{2^{\gamma(1+\varepsilon)}-2}\sum_{i=1}^n C_i.
$$
\end{prop}

The proof is adapted from reference \cite{FP2006}. A sequence $(\mu_n)_{n\ge 0}$ of continuous maps from $P \times P$ into $B$ is defined by $\mu_0=\mu$ and
\begin{equation}
\label{mun}
	\mu_n(s,t):=\sum_{i=0}^{2^n-1}\mu(t_i,t_{i+1})
\end{equation}
where $(t_i)$ is a $2^n$-chain between $s$ and $t$ obtained by iterating the $\gamma$-regularity property $n$ times. Denoting by $C$ the sum $C_1+\cdots+ C_n$, the estimate
$$
	\vert\mu_{n+1}(s,t)-\mu_n(s,t)\vert\le C2^{-\gamma(1+\varepsilon)}2^{n[1-\gamma(1+\varepsilon)]}d_P(s,t)^{1+\varepsilon}
$$
holds, which is easily seen by expressing $\mu_{n+1}(s,t)-\mu_n(s,t)$ as the sum of the $2^n$ terms 
$$\mu(t_{2k},t_{2k+1})+\mu(t_{2k+1},t_{2k+2})-\mu(t_{2k},t_{2k+2}),\ k=0,\ldots, 2^n-1$$
and applying \eqref{fdpmu} to each of them. Hence the sequence $(\mu_n)_{n\ge 0}$ is Cauchy in the complete metric space $\Cal C(P^2,B)$ of continuous maps from $P \times P$ into $B$ endowed with the uniform convergence norm:
$$
	\|f\|:=\mop{sup}_{(s,t)\in[S,T]^2}\|f(s,t)\|_B,
$$
and thus converges uniformly to a limit $\Phi$, which moreover verifies:
\begin{equation}
\label{est-phi}
	\vert\Phi(s,t)-\mu(s,t)\vert\le 2^{-\gamma(1+\varepsilon)}Cd_P(s,t)^{1+\varepsilon}
	\sum_{n\ge 0}2^{-n[\gamma(1+\varepsilon)-1]}=Cd_P(s,t)^{1+\varepsilon}\frac{1}{2^{\gamma(1+\varepsilon)}-2}.
\end{equation}

\begin{lem}
The limit $\Phi$ is independent of the choice of the $2^n$-chain at each step and is additive, that is, it satisfies
$$
	\Phi(s,t)=\Phi(s,u)+\Phi(u,t)
$$
for any $s,u,t\in P$.
\end{lem}

\begin{proof}
\textcolor{blue}{To be changed!} Let $s,t\in P$ with $s\neq t$ and, for any $n\ge 0$, let $\mu_n(s,t)$ and $\mu'_n(s,t)$ be two sequences defined by
\begin{equation}\label{mun-munprime}
	\mu_n(s,t):=\sum_{i=0}^{2^n-1}\mu(t_i,t_{i+1}),\hskip 12mm \mu'_n(s,t):=\sum_{i=0}^{2^n-1}\mu(t'_i,t'_{i+1})
\end{equation}
where $(t_i)$ and $(t'_i)$ are two $2^n$-chains between $s$ and $t$ obtained by iterating the $\gamma$-regularity property $n$ times. Both sequences are Cauchy and converge to their respective limit $\Phi(s,t)$ and $\Phi'(s,t)$. Now we have for any $n\ge 0$:
\begin{eqnarray*}
\vert \mu_n(s,t)-\mu'_n(s,t)\vert&\le& \sum_{i=0}^{2^n-1}\vert\mu(t_i,t_{i+1})-\mu(t'_i,t'_{i+1})\vert\\
&=&
\end{eqnarray*}

Let $u\in P$ such that $d_P(s,u)\le 2^{-\gamma}d_P(s,t)$ and $d_P(s,u)\le 2^{-\gamma}d_P(s,t)$. From $\mu_{n+1}(s,t)=\mu_n(s,u)+\mu_n(u,t)$ we get
$$
	\Phi(s,t)=\Phi(s,u)+\Phi(u,t)
$$
for any $s,t\in P$ and any $u\in P$ as above. Moreover, $\Phi$ is the unique semi-additive map satisfying estimates \eqref{est-phi}. Indeed, if $\Psi$ is another one, then
\begin{eqnarray*}
	\vert (\Phi-\Psi)(s,t)\vert
				&=&\left\vert\sum_{i=0}^{2^n-1}(\Phi-\Psi)(t_{i+1}-t_i)\right\vert\\
				&\le&2C'\sum_{i=0}^{2^n-1}\vert t_{i+1}-t_i\vert^{1+\varepsilon}\\
				&\le& 2C'\vert t-s\vert 2^{-n\varepsilon}
\end{eqnarray*}
with $C'=C/(2^{1+\varepsilon}-2)$. Hence $\Psi=\Phi$ by letting $n$ go to infinity. Now, if $r$ is any positive integer, then the map $\Psi_r$ defined by
$$
	\Psi_r(s,t)=\sum_{j=0}^{r-1}\Phi(t_j,t_{j+1}),
$$
with $t_j=s+j(t-s)/r$, is also semi-additive, hence $\Psi_r=\Phi$. From this we easily get
$$
	\Phi(s,t)=\Phi(s,u)+\Phi(u,t)
$$
for any rational barycenter $u$ of $s$ and $t$, i.e., such that $u=as+(1-a)t$ with $a\in[0,1]\cap\mathbb Q$. Additivity of $\Phi$ follows from continuity.
\end{proof}

The proof of Proposition \ref{prop:fdpsewing} follows by choosing $\varphi(t):=\Phi(o,t)$ for any fixed but arbitrary $o\in[S,T]$. Uniqueness of $\varphi$ up to an additive constant follows immediately from the uniqueness of the additive function $\Phi$ satisfying estimate \eqref{est-phi}.

%%%%%%%%%%%%%%%%%%%%%%%%%%%%%%%%
%%%%%%%%%%%%%%%%%%%%%%%%%%%%%%%%

\item Replace $C|t-s|$ by a more general \textsl{control} $\omega_{st}$, and replace $t\mapsto t^{1+\varepsilon}$ by a more general \textsl{remainder function} $\varpi$ (see \cite[Paragraph 1.2.2]{BL2018a}).
\item Try to adapt M. Gubinelli's cohomological approach \cite{Gubi2010} to this context.
\item ...
\end{enumerate}
}

%%%%%%%%%%%%%%%%%%%%%%%%%%%%%%%%
%%%%%%%%%%%%%%%%%%%%%%%%%%%%%%%%
%%%%%%%%%%%%%%%%%%%%%%%%%%%%%%%%
%%%%%%%%%%%%%%%%%%%%%%%%%%%%%%%%

\end{document}